\documentclass[11pt, lmargin=1.25in, rmargin=1.25in]{article}
\RequirePackage{amsthm,amsmath,amsfonts,amssymb}
\RequirePackage[authoryear]{natbib} 
\RequirePackage[colorlinks,citecolor=blue,urlcolor=blue]{hyperref}

\usepackage{bm}
\usepackage{mathrsfs}  
\usepackage{appendix}
\usepackage{commath}
\usepackage{dsfont}
\usepackage{subfig}
\usepackage{xcolor}
\usepackage{soul}

\usepackage{mathtools}
\usepackage{enumitem}
\usepackage{fancyhdr}
\usepackage{mathrsfs} 

\usepackage{geometry}

\usepackage{setspace}
\newcommand{\N}{\mathbb{N}}
\newcommand{\R}{\mathbb{R}}
\newcommand{\eps}{\varepsilon}

\theoremstyle{plain}
\newtheorem{theorem}{Theorem}[section]

\newtheorem{proposition}[theorem]{Proposition}

\theoremstyle{remark}
\newtheorem{remark}{Remark}[section]

\begin{document}

\title{Superconsistency of Tests in High Dimensions\footnote{We are grateful to the comments of the Editor and four referees which helped to improve the previous version of the manuscript.}}

\date{First version: May 2021 \\
This version: January 2022}

\author{
\begin{tabular}{c}
Anders Bredahl Kock \\ 
\small	University of Oxford \\
\small	CREATES, Aarhus University\\
\small	10 Manor Rd, Oxford OX1 3UQ
\\
\small	{\small	\href{mailto:anders.kock@economics.ox.ac.uk}{anders.kock@economics.ox.ac.uk}} 
\end{tabular}
\and
\begin{tabular}{c}
David Preinerstorfer \\ 
{\small	SEW and SEPS} \\ 
{\small	 University of St.~Gallen} \\
{\small	Varnb\"uelstrasse 14, St.~Gallen} \\ 
{\small	 \href{mailto:david.preinerstorfer@unisg.ch}{david.preinerstorfer@unisg.ch}}
\end{tabular}
}

\maketitle	

\begin{abstract}
To assess whether there is some signal in a big database, aggregate tests for the global null hypothesis of no effect are routinely applied in practice before more specialized analysis is carried out. Although a plethora of aggregate tests is available, each test has its strengths but also its blind spots. In a Gaussian sequence model, we study whether it is possible to obtain a test with substantially better consistency properties than the likelihood ratio (i.e., Euclidean norm based) test. We establish an impossibility result, showing that in the high-dimensional framework we consider, the set of alternatives for which a test may improve upon the likelihood ratio test --- that is, its superconsistency points --- is always asymptotically negligible in a relative volume sense. \end{abstract}

\newpage

\section{Introduction}

A major challenge in the current ``big-data era'' is to extract signals from huge databases. Often, an applied researcher proceeds in a two-step fashion: First, in order to decide whether there is any signal in the data at all, one performs an aggregate test of the global null hypothesis of no signal. This global null hypothesis is typically formulated as the high-dimensional target parameter being the zero vector. Second, if the global null hypothesis was rejected by the test, further analysis is undertaken to uncover the precise nature of the signal. Much research has been directed to studying properties of such a sequential rejection principle, cf.~\cite{romano2005exact}, \cite{yekutieli2008hierarchical}, \cite{rosenbaum2008testing}, \cite{meinshausen2008hierarchical}, \cite{goeman2010sequential}, \cite{heller2}, \cite{bogomolov2020hypotheses} and references therein. 

Using a powerful test for the global null hypothesis in the first step of such a hierarchical multi-step procedure is of course crucial, and the development of tests for this hypothesis has therefore attracted much research in its own right. A typical choice, employed in, e.g., \cite{heller}, is to use a test based on the Euclidean norm of the estimator. This also leads to the likelihood ratio (LR) test in the Gaussian sequence model they considered, which is also the framework in the present article. Although the LR test is a natural choice, one may ask: \emph{Do tests for the global null exist that are consistent against substantially more alternatives than the LR test?} This question is practically relevant, because one can choose from a large menu of well-established tests, yet precisely which one to use is not obvious: For example, one could use tests based on other norms than the Euclidean one, a natural class of tests being based on~$p$-norms, cf.~the classic monograph of~\cite{ingster}. One could also use a test based on combining different~$p$-norms as suggested by the power enhancement principle of \cite{fan2015} and in~\cite{kp2}. The possibility of increasing power by combining tests has recently been applied in many types of high-dimensional testing problems, cf.~\cite{xu2016adaptive}, \cite{yang2017weighted}, \cite{yu2020fisher}, \cite{he2021asymptotically}, \cite{yu2021power} [testing high-dimensional means and covariance matrices]; \cite{zhang2021adaptive} [change point detection]; \cite{jammalamadaka2020sobolev} [tests for uniformity on the sphere]; \cite{feng2020max} [tests for cross-sectional independence
in high-dimensional panel data models]. Another test that has gained popularity in recent years is the Higher Criticism. This test dates back to~\cite{tukey1976t13} and its strong power properties against deviations from the global null were first exhibited by~\cite{donoho} and have led to much subsequent research, cf.~\cite{donoho2009feature}, \cite{hall2010innovated}, \cite{ tony2011optimal}, \cite{arias2011global}, \cite{barnett2014analytical}, \cite{li2015higher}, \cite{arias2019detection} and \cite{ porter2020beyond}. Alternatively, one could use tests based on combining~p-values for coordinate-wise zero restrictions. Important early work includes~\cite{fisher1934statistical}, \cite{tippett1931methods}, \cite{pearson1933method}, \cite{stouffer1949american} and \cite{simes1986improved}. For a review of the classic literature see \cite{cousins2007annotated}, more recent contributions are \cite{owen09},  \cite{duan2020interactive} and~\cite{vovk2020combining, vovk2020values}. It is crucial to highlight here that many of the above mentioned tests are consistent against strictly more alternatives than the LR test, i.e., they dominate the LR test in terms of their consistency properties; indeed, this is the main motivation of the power enhancement principle. Hence, the question of interest in the present article is not whether one can do better than the LR test at all, but whether one can do \emph{substantially} better.

We consider the question raised in the previous paragraph from a high-dimensional perspective. In the Gaussian sequence model, we investigate
whether aggregate tests can be obtained that are consistent against substantially more alternatives than the likelihood
ratio test. We show that relative to a uniform prior on the parameter space this is impossible: essentially, we prove that for any given test the set of alternatives against which it is consistent, but the LR test is not, has vanishing relative Lebesgue measure. Hence, no test for the global null hypothesis can substantially improve on the LR test. From a technical perspective, our proofs are based on results by~\cite{ss} concerning the asymptotic volume of intersections of~$p$-norm balls and on the concentration phenomenon for Lipschitz continuous functions on spheres as exposited in~\cite{braz} or \cite{vershynin_2018}.

Our finding is reminiscent of \cite{lecam1953}, who showed (in finite-dimensional settings) that the set of possible superefficiency points of an estimator relative to the maximum likelihood estimator cannot be larger than a Lebesgue null set; cf.~also \cite{vanderVaart1997}. Note that our result does not imply that one should always use the LR test and not think carefully about the choice of test in high-dimensional testing problems. If, for example, one is interested in particular types of deviations from the null, e.g., sparse ones, there may be good reasons to use a test based on the supremum norm or the Higher Criticism. Furthermore, albeit very natural, the magnitude of the consistency set is merely one of many properties that can be used to compare tests. For example, tests are also frequently compared in terms of, e.g., their minimax detection properties or their local power against deviations from the null of a specific type. Nevertheless, in analogy to \cite{lecam1953}, regardless of how clever an alternative test is designed, the amount of alternatives against which one achieves an improvement as compared to the LR test cannot be substantial in terms of relative volume. This also supports basing a combination procedure, such as the power enhancement principle by~\cite{fan2015}, on the Euclidean norm.

\section{Framework and terminology}

We consider the Gaussian sequence model
\begin{equation}\label{eqn:model}
y_{i,d} = \theta_{i,d} + \varepsilon_i, \quad i = 1, \hdots, d,
\end{equation}
where~$y_{1,d}, \hdots, y_{d,d}$ are the observations, the parameters~$\theta_{i,d} \in \R$ are unknown, and where the unobserved terms~$\varepsilon_i$ are independent and standard normal. Writing~$\bm{y}_d = (y_{1,d}, \hdots y_{d,d})'$, $\bm{\varepsilon}_d = (\varepsilon_1, \hdots, \varepsilon_d)'$, and~$\bm{\theta}_d = (\theta_{1,d}, \hdots, \theta_{d,d})' \in \R^d$, one can equivalently state the model in~\eqref{eqn:model} as~$\bm{y}_d=\bm{\theta}_d+\bm{\eps}_d$. We observe a single realization of a~$d$-dimensional Gaussian vector~$\bm{y}_d$ with mean~$\bm{\theta}_d$ and identity covariance matrix for each $d\in\N$. In this sense the ``sample size'' is one for each~$d$ (but cf.~Remark~\ref{rem:concGS} below). The asymptotic analysis in the Gaussian sequence model then relies on~$d\to\infty$. This is a high-dimensional regime in the sense that the number of parameters,~$d$, tends to infinity. In the model~\eqref{eqn:model}, we are interested in the testing problem
\begin{equation}\label{eqn:tp}
H_{0, d}: \bm{\theta}_d = \bm{0}_d \quad \text{ against } \quad H_{1, d}: \bm{\theta}_d \in \R^d \setminus \{\bm{0}_d\},
\end{equation}
where~$\bm{0}_d$ denotes the origin in~$\R^d$. The null hypothesis~$H_{0, d}$ is typically referred to as the ``global null'' of no effect. 

\begin{remark}\label{rem:concGS}
Although the Gaussian sequence model is an idealization, many fundamental issues of high dimensionality show up already here and insights obtained within this model carry over, at least on a conceptual level, to many other settings. It is therefore widely recognized as an important prototypical framework in high-dimensional statistics, see, for example,~\cite{ingster}, \cite{carpentier2019adaptive}, \cite{johnstone} or \cite{castillo2020spike}. To make this more precise, consider a situation where an estimator~$\hat{\bm{\beta}}_d$ for a target parameter~$\bm{\beta}_d \in \R^d$ is available the distribution of which satisfies
\begin{equation}\label{eq:approxdist}
\hat{\bm{\beta}}_d \approx \mathbb{N}(\bm{\beta}_d, \bm{\Omega}_d).
\end{equation}
Suppose further that an invertible estimator~$\hat{\bm{\Omega}}_d \approx \bm{\Omega}_d$ is at one's disposal, such that
\begin{equation*}
\hat{\bm{\Omega}}_d^{-1/2} \hat{\bm{\beta}}_d \approx \mathbb{N}(\bm{\Omega}_d^{-1/2}\bm{\beta}_d, \bm{I}_d).
\end{equation*}
Then, testing~$\bm{\beta}_d = \bm{0}_d$ on the basis of~$\hat{\bm{\beta}}_d$ and~$\hat{\bm{\Omega}}_d$ is approximated by testing~$\bm{\theta}_d :=  \bm{\Omega}_d^{-1/2}\bm{\beta}_d = \bm{0}_d$ in a Gaussian sequence model. Precise sets of conditions under which the above approximation statements hold depend on the interplay of the dimension of the target parameter and sample size as well as particularities of the specific setup under consideration. This has been a topic of intense research interest in recent years and suitable sets of sufficient conditions can be found in, e.g., \cite{bentkus2003dependence}, \cite{chernozhukov2017central} and \cite{GF}. Working directly with a Gaussian sequence model allows us to bypass such aspects.
\end{remark}

For a given~$d \in \N$, a (possibly randomized) test~$\varphi_d$, say, for~\eqref{eqn:tp} is a (measurable) function from the sample space~$\R^d$ to the closed unit interval. In the asymptotic framework we consider, we are interested in properties of sequences of tests~$\{\varphi_d\}$, where $\varphi_d$ is a test for~\eqref{eqn:tp} for every~$d \in \N$. To lighten the notation, we shall write~$\varphi_d$ instead of~$\{\varphi_d\}$ whenever there is no risk of confusion. We are particularly interested in the consistency properties of sequences of tests. As usual, we say that a sequence of tests~$\varphi_d$ is consistent against the \emph{array} of parameters~$\bm{\vartheta} = \{\bm{\theta}_d : d \in \N\}$, where~$\bm{\theta}_d \in \R^d$ for every~$d \in \N$, if and only if (as~$d \to \infty$)
\begin{equation*}
\mathbb{E}\left(\varphi_d(\bm{\theta}_d + \bm{\varepsilon}_d)\right) \to 1.
\end{equation*}
To every sequence of tests~$\varphi_d$ we associate its \emph{consistency set}~$\mathscr{C}(\varphi_d)$, say. The consistency set~$\mathscr{C}(\varphi_d)$ is the set of all arrays of parameters~$\bm{\vartheta}$ the sequence of tests~$\varphi_d$ is consistent against. By definition~$$\mathscr{C}(\varphi_d) \subseteq \bigtimes_{d = 1}^{\infty} \R^d =: \bm{\Theta},$$ the latter denoting the set of all possible arrays of parameters. 

Recall that a sequence of tests~$\varphi_d$ is said to have \emph{asymptotic size}~$\alpha \in [0, 1]$ if
\begin{equation*}
\mathbb{E}\left(\varphi_d(\bm{\varepsilon}_d)\right) \to \alpha.
\end{equation*}
In this article, following the Neyman-Pearson paradigm, we focus on the case where~$\alpha \in (0, 1)$, which we shall implicitly assume in the discussions throughout unless mentioned otherwise.

It is well-known that the LR test for~\eqref{eqn:tp} rejects if the Euclidean norm~$\|\cdot\|_2$ of the observation vector~$\bm{y}_d$ exceeds a critical value~$\kappa_{d,2}$ chosen to satisfy the given size constraints. That is, the LR test is given by~$\mathds{1}\{\|\cdot\|_2 \geq \kappa_{d,2}\}$.  For notational simplicity, we abbreviate the sequence of tests~$\{\mathds{1}\{\|\cdot\|_2 \geq \kappa_{d,2}\} \}$ by~$\{2, \kappa_{d,2}\}$ and thus write~$\mathscr{C}(\{2, \kappa_{d,2}\})$ for its consistency set. The following result is contained in \cite{ingster},~cf.~also Theorem~3.1 in \cite{kp2} for extensions. 

\begin{theorem}\label{thm:ing}
Let~$\kappa_{d,2}$ be a sequence of critical values such that the asymptotic size of~$\{2, \kappa_{d,2}\}$ is~$\alpha \in (0, 1)$. Then
\begin{equation}\label{eqn:2normcons}
\bm{\vartheta} \in \mathscr{C}(\{2, \kappa_{d,2}\}) \quad \Leftrightarrow \quad d^{-1/2} \|\bm{\theta}_d\|_2^2 \to \infty.
\end{equation}
\end{theorem}
Theorem~\ref{thm:ing} shows that the consistency set of the LR test is precisely characterized by the asymptotic behavior of the Euclidean norms of the array of alternatives under consideration. That the consistency set of the LR test can be completely characterized in terms of the norm its test statistic is based on seems natural, but is quite specific to the LR test, see Theorem~3.1 and the ensuing discussion in \cite{kp2}.

\section{Superconsistency points}\label{sec:supcp}

\subsection{Improving on the LR test}\label{sec:improving}

Although the LR test is a canonical choice of a test for the testing problem~\eqref{eqn:tp}, there are many other reasonable tests available. For example, classic results by~\cite{birnbaum1955} and~\cite{stein1956} show that any test with convex acceptance region (i.e., the complement of its rejection region) is admissible. \citeauthor{anderson1955integral}'s (\citeyear{anderson1955integral}) theorem implies that if the acceptance region is furthermore symmetric around the origin then the test is also unbiased. Thus, any convex symmetric (around the origin) set delivers an admissible unbiased test, which is hence reasonable from a non-asymptotic point of view. 

One class of tests that is intimately related to the LR tests consists of tests based on other~$p$-norms than the Euclidean one. For~$\bm{x} = (x_1, \hdots, x_d)' \in \R^d$ and~$p \in (0, \infty]$, define the~$p$-norm as usual via\footnote{Strictly speaking,~$||\cdot||_p$ defines a norm on~$\R^d$ only for~$p\in[1,\infty]$ and a quasi-norm for~$p \in (0, 1)$.}
\begin{equation*}
\|\bm{x}\|_p = 
\begin{cases}
\left(\sum_{i = 1}^d |x_i|^p\right)^{\frac{1}{p}} & \text{if } p < \infty, \\
\max_{i = 1, \hdots, d} |x_i| & \text{else}.
\end{cases}
\end{equation*}
In analogy to the LR test,~$p$-\emph{norm based} tests reject if the~$p$-norm of the observation vector exceeds a critical value~$\kappa_{d,p}$. Special cases, which have an established tradition in high-dimensional inference, are the~$1$- and the supremum norm. We shall denote the sequence of tests~$\{\mathds{1}\{\|\cdot \|_p \geq \kappa_{d,p}\}\}$ by~$\{p, \kappa_{d,p}\}$. Clearly,~$p$-norm based tests are unbiased and admissible for~$p\in[1,\infty]$ as a consequence of the discussion in the first paragraph of this section.

Concerning the consistency sets~$\mathscr{C}(\{p, \kappa_{d,p}\})$ of general~$p$-norm based tests, it is a somewhat surprising fact that
\begin{enumerate}[label=(\roman*)]
\item~$\mathscr{C}(\{p, \kappa_{d,p}\})\subsetneqq \mathscr{C}(\{q, \kappa_{d,q}\})$ for~$0 < p < q < \infty$, i.e., strictly larger exponents~$p$ result in \emph{strictly} larger consistency sets; and 
\item that this ranking does not extend to~$q = \infty$, in the sense that there are alternatives the supremum norm based test is not consistent against but against which the LR test is consistent and vice versa;
\end{enumerate}
see~\cite{kp2} for formal statements.\footnote{Recall that throughout the present article we implicitly impose the condition that all tests have asymptotic size in~$(0, 1)$ if not otherwise mentioned.} From (i) it follows that any~$p$-norm based test with~$p \in (2, \infty)$ has a \emph{strictly} larger consistency set than the LR test. We stress that this asymptotic strict domination of the LR test in terms of consistency sets is not in contradiction to its admissibility for each~$d\in\N$.

Other tests that strictly dominate the LR test can be obtained, e.g., through combination procedures that enhance the LR test with a sequence of tests~$\eta_d$ that is sensitive against alternatives of a different ``type'' than the LR test in the sense that~$$\mathscr{C}(\eta_d) \setminus \mathscr{C}(\{2, \kappa_{d,2}\}) \neq \emptyset.$$ To see how this can be achieved, note that the consistency set of the sequence of tests~$\psi_d$, say, where~$\psi_d$ rejects if the LR test \emph{or}~$\eta_d$ rejects, contains~$\mathscr{C}(\{2, \kappa_{d,2}\}) \cup \mathscr{C}(\eta_d)$, and hence dominates the LR test in terms of consistency. Essentially, this is the power enhancement principle of~\cite{fan2015}; see~\cite{kp1} for related results. Note that if~$\eta_d$ has asymptotic size~$0$, which is an assumption imposed on~$\eta_d$ in the context of the power enhancement principle, nothing is lost in terms of asymptotic size when using~$\psi_d$ instead of the LR test, because both sequences of tests then have the same asymptotic size.\footnote{If~$\eta_d$ has a positive asymptotic size that is smaller than the asymptotic size targeted in the final combination test, one can work with a LR test with small enough asymptotic size in the combination procedure to obtain a test that dominates the LR test in terms of consistency (recall from Theorem~\ref{thm:ing} that the consistency set of the LR test does not depend on the specific value of the asymptotic size).}
 
To clarify how much can possibly be gained in terms of consistency by using a sequence of tests~$\varphi_d$ other than the LR test, we shall consider the corresponding set~$$\mathscr{C}(\varphi_d) \setminus \mathscr{C}(\{2, \kappa_{d,2}\}),$$ which we refer to as the \emph{superconsistency points} of the sequence of tests~$\varphi_d$ (relative to the LR test). Note that the set of superconsistency points is defined for any sequence of tests, regardless of whether it dominates the LR test or not (in the sense that its consistency set includes that of the LR test).\footnote{To provide an example, for any $p \in (2,\infty)$ the set of superconsistency points of the $p$-norm based test is fully characterized by Corollary 3.2 in Kock and Preinerstorfer 	(2021), cf.~also their Theorem 3.4 which essentially shows that these superconsistency points are approximately sparse and have at least one large entry.} On a conceptual level, superconsistency points are related to superefficiency points of estimators relative to the maximum likelihood estimator in classic parametric theory.

\subsection{The relative volume of the set of superconsistency points}\label{sec:relvol}

The central question we consider in this article is how ``large'' the set of superconsistency points~$\mathscr{C}(\varphi_d) \setminus \mathscr{C}(\{2, \kappa_{d,2}\})$ can possibly be for a sequence of tests~$\varphi_d$ with asymptotic size in~$(0, 1)$.  Note that the larger~$\mathscr{C}(\varphi_d) \setminus \mathscr{C}(\{2, \kappa_{d,2}\})$ is, the larger is the set of alternatives the sequence of tests~$\varphi_d$ is consistent against but the LR test is not consistent against.  Although we already know from the examples discussed in Section~\ref{sec:improving} that~$\mathscr{C}(\varphi_d) \setminus \mathscr{C}(\{2, \kappa_{d,2}\})$ is non-empty for many~$\varphi_d$, we here investigate whether one can \emph{substantially} enlarge the consistency set by using another test than the LR test. 

To make the above question amenable to a formal treatment, note that Theorem~\ref{thm:ing} implies that for any sequence of LR tests~$\{2, \kappa_{d,2}\}$ with asymptotic size~$\alpha \in (0, 1)$, the complement of~$\mathscr{C}(\{2, \kappa_{d,2}\})$ satisfies 
\begin{equation*}
\bm{\Theta} \setminus \mathscr{C}(\{2, \kappa_{d,2}\}) \supseteq \bigtimes_{d = 1}^{\infty} \mathbb{B}_2^d(r_d)
\end{equation*}
if the sequence~$r_d > 0$ is such that~$r_d/d^{1/4}$ is bounded and
where~$\mathbb{B}_2^d(r)$ denotes the Euclidean ball with radius~$r$ centered at the origin. That is, the LR test is inconsistent against any element of~$\bigtimes_{d = 1}^{\infty} \mathbb{B}_2^d(r_d)$. We now investigate how many inconsistency points of the LR test can be removed from any such benchmark~$\bigtimes_{d = 1}^{\infty} \mathbb{B}_2^d(r_d)$ by erasing all superconsistency points of a sequence of tests~$\varphi_d$. 

\begin{figure}
\centering
\includegraphics[width=0.7\linewidth]{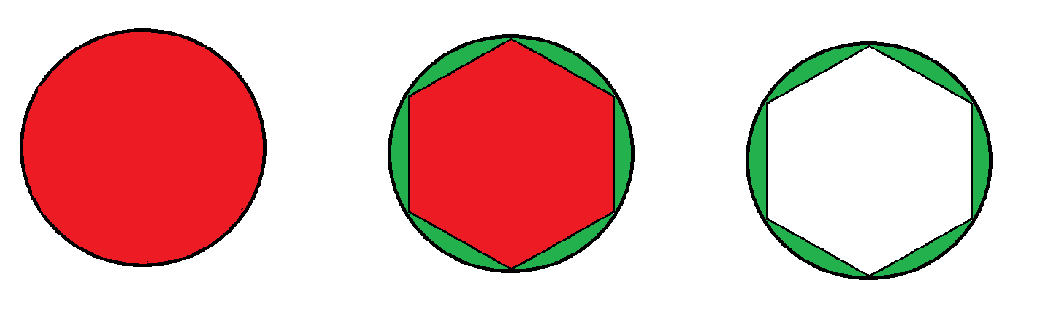}
\caption{Illustration of $\mathbb{B}_2^d(r_d)$ (red), $\mathbb{D}_d \subseteq \mathbb{B}_2^d(r_d)$ and $\mathbb{D}_d$ (green) for~$d = 2$.}
\label{fig:bitet}
\end{figure}

Formally, this is to be understood in the following sense: let~$\varphi_d$ be a sequence of tests with consistency set~$\mathscr{C}(\varphi_d)$ and let~$r_d$ be such that~$r_d/d^{1/4}$ is bounded. Let~$\mathbb{D}_d \subseteq  \mathbb{B}_2^d(r_d)$ be such that~$$\bigtimes_{d = 1}^{\infty} \mathbb{D}_d \subseteq \mathscr{C}(\varphi_d).$$ Note that all elements of~$\bigtimes_{d = 1}^{\infty} \mathbb{D}_d$ are superconsistency points of~$\varphi_d$ which are also contained in the benchmark~$\bigtimes_{d = 1}^{\infty} \mathbb{B}_2^d(r_d)$ (cf.~the illustration in Figure~\ref{fig:bitet}). Denoting by~$\text{vol}_d$ the~$d$-dimensional Lebesgue measure, we investigate the asymptotic behavior of the~\emph{relative} volume measure 
\begin{equation}\label{eqn:limc1p}
\frac{\mathrm{vol}_d\left(\mathbb{D}_d\right)}{\mathrm{vol}_d\left(\mathbb{B}_2^d(r_d)\right)}.
\end{equation} 
Obviously, the ratio in~\eqref{eqn:limc1p} is a number in~$[0, 1]$. On the one hand, if this ratio is asymptotically close to~$1$, this means that, in terms of relative volume, many elements of the benchmark~$\bigtimes_{d = 1}^{\infty} \mathbb{B}_2^d(r_d)$ are superconsistency points of the sequence of tests~$\varphi_d$. That is, one can \emph{substantially} improve upon the LR test by using~$\varphi_d$ (or by combining the LR test with~$\varphi_d$ through the power enhancement principle). On the other hand, if this ratio is asymptotically close to~$0$, this means that in terms of relative volume only few elements of the benchmark are superconsistency points of~$\varphi_d$. 

\begin{remark}
One could also study the asymptotic behavior of the sequences $\mathrm{vol}_d\del[1]{\mathbb{B}_2^d(r_d)} - \mathrm{vol}_d\left(\mathbb{D}_d\right)$ or~$\mathrm{vol}_d\left(\mathbb{D}_d\right)$ in order to determine whether one can substantially improve upon the LR test. However, these sequences both converge to~$0$. To see this, just note that
\begin{equation*}
	\mathrm{vol}_d\left(\mathbb{B}_2^d(r_d)\right) = \frac{\pi^{d/2}}{\Gamma(d/2 + 1)} r_d^d \to 0,
\end{equation*}
in case~$r_d/d^{1/4}$ is bounded as a consequence of Stirling's approximation to the gamma function as well as~$\mathbb{D}_d \subseteq  \mathbb{B}_2^d(r_d)$. Thus, such ``absolute'' volume measures are uninformative, since even the absolute volume of~$\mathbb{B}_2^d(r_d)$ tends to zero.
\end{remark}

\begin{remark}
One may argue that rather than~\eqref{eqn:limc1p} one should study~$$\frac{\mathrm{vol}_d\left(\mathbb{D}_d\right)}{\mathrm{vol}_d(\mathrm{proj}_d[\bm{\Theta} \setminus \mathscr{C}(\{2, \kappa_{d,2}\})])},$$ $\mathrm{proj}_d(\cdot)$ denoting the projection onto the~$d$th coordinate of its argument. However, since~$$\bigtimes_{d=1}^K\R^d\times \bigtimes_{d=K+1}^\infty\cbr[0]{\bm{0}_d} \subseteq\bm{\Theta} \setminus \mathscr{C}(\{2, \kappa_{d,2}\})$$ for all~$K\in\N$, it follows that~$\mathrm{vol}_d(\mathrm{proj}_d[\bm{\Theta} \setminus \mathscr{C}(\{2, \kappa_{d,2}\})])=\infty$ for all~$d\in\N$.
\end{remark}

We emphasize that using the (normalized) Lebesgue measure to assess the asymptotic magnitude of the set of superconsistency points is one among many possible choices. Other measures would be possible too, but the uniform prior over~$\mathbb{B}_2^d(r_d)$ is a natural choice as in many situations there is no clear guidance concerning the type of alternative one wishes to favor.\footnote{Our results remain valid if, instead of measuring the magnitude of~$\mathbb{D}_d$ w.r.t.~the uniform probability measure on~$\mathbb{B}_2^d(r_d)$, one measures its magnitude w.r.t.~the uniform probability measure on the  Euclidean sphere of radius~$r_d$. We will comment on this in Remark~\ref{rem:surf}, but will focus on the uniform distribution on~$\mathbb{B}_2^d(r_d)$ throughout the article.}

Note that the ratio in~\eqref{eqn:limc1p} depends on two ingredients:
\begin{enumerate}
\item the benchmark~$\bigtimes_{d = 1}^{\infty} \mathbb{B}_2^d(r_d)$;  
\item the sequence of superconsistency points~$\bigtimes_{d = 1}^{\infty} \mathbb{D}_d$ which depends on the sequence of tests~$\varphi_d$.
\end{enumerate}
Therefore, one could suspect that the asymptotic behavior of~\eqref{eqn:limc1p} depends in a complicated way on the interplay between these two components. Nevertheless, it turns out that the asymptotic behavior of~\eqref{eqn:limc1p} has a simple description that does not depend on any of the two ingredients just described. In fact, we shall prove in Section~\ref{sec:general} that the limit of the sequence is~$0$ for \emph{all} sequences of tests~$\varphi_d$. Hence, it is impossible to improve on the LR test in terms of the magnitude of its consistency set apart from a set of superconsistency points that is negligible in a relative volume sense.

In the following Section~\ref{sec:pproof}, we shall first establish this result for~$\varphi_d$ a sequence of~$p$-norm based tests with~$p \in (2, \infty)$.  Note that all these tests have a strictly larger consistency set than the LR test as discussed in Section~\ref{sec:improving}. A general result, the proof of which is a bit more involved, will be presented in Section~\ref{sec:general}.

\section{$p$-norm based tests}\label{sec:pproof}
We now consider the asymptotic behavior of the sequence~\eqref{eqn:limc1p} for the special case where~$\varphi_d$ is a sequence of~$p$-norm based tests with~$p \in (2, \infty)$ being fixed. For this class of tests, we can exploit the characterization of their consistency sets provided in Theorem 3.1 and Corollary 3.2 of~\cite{kp2}, together with results from asymptotic geometry developed in~\cite{ss} based on earlier results in~\cite{sz}. These ingredients lead to a direct proof of the limit of the sequence in~\eqref{eqn:limc1p} being~$0$. 
\begin{theorem}\label{thm:nogain}
Let~$p \in (2, \infty)$ and let the sequence of critical values~$\kappa_{d,p}$ be such that~$\{p, \kappa_{d,p}\}$ has asymptotic size~$\alpha \in (0, 1)$. Then, for any sequence~$r_d > 0$ such that~$r_d/d^{1/4}$ is bounded, and any sequence of non-empty Borel sets~$\mathbb{D}_d \subseteq \mathbb{B}_2^d(r_d)$ such that~
\begin{equation}\label{eqn:dpf}
\bigtimes_{d = 1}^{\infty} \mathbb{D}_d \subseteq \mathscr{C}(\{p, \kappa_{d,p}\}),
\end{equation}
we have
\begin{equation*}
\lim_{d \to \infty} \frac{\mathrm{vol}_d\left(\mathbb{D}_d\right)}{\mathrm{vol}_d\left(\mathbb{B}_2^d(r_d)\right)} = 0.
\end{equation*} 
\end{theorem}

\begin{proof}
Let~$\{p, \kappa_{d,p}\}$,~$r_d$, and~$\mathbb{D}_d$ be as in the statement of the theorem. Corollary~3.2 in~\cite{kp2} shows that~$\bm{\vartheta} \in \mathscr{C}(\{p, \kappa_{d,p}\})$ if and only if~$d^{-1/2} (\|\bm{\theta}_d\|_2^2 \vee \|\bm{\theta}_d\|_p^p) \to \infty$. Together with~$r_d/d^{1/4}$ being bounded, ~$\mathbb{D}_d \subseteq \mathbb{B}_2^d(r_d)$, and~\eqref{eqn:dpf}, this guarantees that~$\tilde{s}_d/d^{1/(2p)} \to \infty$ for~$\tilde{s}_d := \inf\{ \|\bm{\theta}_d\|_p: \bm{\theta}_d \in \mathbb{D}_d \}$. 
The definition of~$\tilde{s}_d$ implies~$$\mathbb{G}_d := \mathbb{B}_2^d(r_d) \setminus \mathbb{D}_d \supseteq \mathbb{B}_2^d(r_d) \cap \mathbb{B}_p^d(\tilde{s}_d/2).$$
Define the sequence~$s_d := d^{1/(2p) - 1/4} r_d > 0$, so that~$s_d/d^{1/(2p)} = r_d/d^{1/4}$ is bounded. Hence, eventually~$\tilde{s}_d \geq 2s_d$ and thus~$\mathrm{vol}_d(\mathbb{G}_d) \geq \mathrm{vol}_d(\mathbb{B}_2^d(r_d) \cap \mathbb{B}_p^d(s_d))$ holds, so that the quotient~
\begin{equation*}
1 - \frac{\mathrm{vol}_d\left(\mathbb{D}_d\right)}{\mathrm{vol}_d(\mathbb{B}_2^d(r_d))} = \frac{\mathrm{vol}_d\left(\mathbb{G}_d\right)}{\mathrm{vol}_d(\mathbb{B}_2^d(r_d))} 
\end{equation*}
is eventually not smaller than
\begin{equation*}
\frac
{\mathrm{vol}_d\left(
\mathbb{B}_2^d(r_d) \cap \mathbb{B}_p^d(s_d)
\right)}
{\mathrm{vol}_d\left(
\mathbb{B}_2^d(r_d)
\right)} = \frac
{\mathrm{vol}_d\left(
\mathbb{B}_2^d(e_{d,2}) \cap \mathbb{B}_p^d(e_{d,2} s_d/r_d)
\right)}
{\mathrm{vol}_d\left(
\mathbb{B}_2^d(e_{d,2})
\right)}= \mathrm{vol}_d
\left(
\mathbb{B}_2^d(e_{d,2}) \cap u_d  \mathbb{B}_p^d(e_{d,p})
\right),
\end{equation*}
where~$u_d := \frac{e_{d,2}}{e_{d,p}} \frac{d^{1/(2p)}}{d^{1/4}}$,~$e_{d,p} := \frac{1}{2} \frac{\Gamma(1+d/p)^{1/d}}{\Gamma(1+1/p)}$, and consequently~$\mathrm{vol}_d(
\mathbb{B}_2^d(e_{d,2})
) = 1$. The main result in \cite{ss} shows that for every~$t$ large enough~$\mathrm{vol}_d
(
\mathbb{B}_2^d(e_{d,2}) \cap t \mathbb{B}_p^d(e_{d,p})
) \to 1$, as~$d \to \infty$. Therefore, we are done upon verifying that~$u_d \to \infty$. This follows from the lower bound
\begin{equation*}
\frac{e_{d,2}}{e_{d,p}} =
\left[\frac{\Gamma(1+d/2)}{\Gamma(1+d/p)}\right]^{1/d}
\frac{\Gamma(\frac{1}{p} + 1)}{\Gamma(\frac{1}{2} + 1)} \geq
\left[d/p\right]^{1/2-1/p}
\frac{\Gamma(\frac{1}{p} + 1)}{\Gamma(\frac{1}{2} + 1)}
,
\end{equation*}
where we used the inequality for ratios involving the gamma function in Equation~12 of~\cite{jameson} with~``$x = 1+d/p$'' (which is not smaller than~$1$) and~``$y = d(1/2 - 1/p)$'' (which is not smaller than~$0$).
\end{proof}

Hence, even though~$\mathscr{C}(\{p, \kappa_{d, p}\})$ contains the consistency set of the LR test as a \emph{strict} subset for every~$p \in (2, \infty)$ as discussed in Section~\ref{sec:improving}, the subset of those alternatives in each benchmark~$\bigtimes_{d = 1}^{\infty} \mathbb{B}_2^d(r_d)$ for which the test~$\{p, \kappa_{d, p}\}$ provides an improvement over the LR test is ``negligible'' in (relative) volume. That this result is not specific to~$p$-norm based tests, but extends to \emph{all} tests will be shown next.

\section{Unrestricted sequences of tests}\label{sec:general}

The proof of Theorem~\ref{thm:nogain} builds heavily on the particular structure of the consistency set of~$p$-norm based tests. We shall now establish that \emph{no} test can improve substantially on the LR test. In the absence of any structure on the tests, one can no longer exploit specific properties of the consistency set stemming from the test being based on a~$p$-norm. Instead we rely on concentration results for Lipschitz continuous functions on the sphere as exposited in~\cite{braz} or \cite{vershynin_2018}.
\begin{theorem}\label{thm:nogain2}
For every sequence of tests~$\psi_d$ with asymptotic size~$\alpha \in (0, 1)$ and every sequence~$r_d > 0$ such that~$r_d/d^{1/4}$ is bounded, there exists an~$\epsilon > 0$, such that for every sequence of non-empty Borel sets~$\mathbb{D}_d \subseteq \mathbb{B}_2^d(r_d)$ satisfying
\begin{equation}\label{eqn:inclconsset}
\bigtimes_{d = 1}^{\infty} \mathbb{D}_d \subseteq \mathscr{C}(\psi_d),
\end{equation} 
we have
\begin{equation}\label{eqn:rate}
\frac{\mathrm{vol}_d\left(\mathbb{D}_d\right)}{\mathrm{vol}_d\left(\mathbb{B}_2^d(r_d)\right)} \leq 2 \exp\left( - \epsilon^2 (d-1)/ (2r_d^2) \right) \quad \text{ for all } d \text{ large enough;}
\end{equation}
in particular~$\mathrm{vol}_d\left(\mathbb{D}_d\right)/\mathrm{vol}_d(\mathbb{B}_2^d(r_d))$ converges to~$0$ as~$d \to \infty$.
\end{theorem}

The proof of Theorem~\ref{thm:nogain2} can be found in Appendix~\ref{proofmain0s}. Note that Theorem~\ref{thm:nogain2} not only shows that the magnitude of superconsistency points of tests is asymptotically negligible for any test --- it also shows that the measure of these points converges to zero quickly in the dimension~$d$.

\begin{remark}[Spherical measure instead of relative volume]\label{rem:surf}
One could ask what happens in the context of Theorem~\ref{thm:nogain2} if, instead of considering~$\mathrm{vol}_d(\mathbb{D}_d)/\mathrm{vol}_d(\mathbb{B}_2^d(r_d))$ in~\eqref{eqn:rate}, one considers~$\rho_{d,r_d}(\mathbb{D}_d)$, where~$\rho_{d,r_d}$ denotes the uniform probability measure on the sphere~$$\mathbb{S}^{d-1}(r_d) := \{\xi \in \R^d: \|\xi\|_2 = r_d\}.$$ Inspection of the proof of Theorem~\ref{thm:nogain2} (cf.~Equation~\eqref{eqn:sphD}) shows that the statement equally holds with~$\mathrm{vol}_d(\mathbb{D}_d)/\mathrm{vol}_d(\mathbb{B}_2^d(r_d))$ replaced by~$\rho_{d,r_d}(\mathbb{D}_d)$. That is, also with this alternative measure, one reaches the same conclusion concerning the magnitude of the set of superefficiency points of a sequence of tests relative to the LR test.
\end{remark}

So far, all our results concerned consistency properties of tests. We were interested in the possible magnitude of the superconsistency points of a sequence of tests relative to the LR test and have seen that the magnitude of such points cannot be substantial. Although we now know that one cannot substantially improve on the LR test in terms of consistency (in the sense of Theorem~\ref{thm:nogain2}), there could in principle exist sequences of tests that have larger power than the LR test on substantial portions of the parameter space (without the power there being close to~$1$). A non-asymptotic question one can therefore ask is: how large can such portions of the parameter space be? To answer this question, we introduce some more notation: let~$\alpha \in [0, 1]$ and denote for every~$r > 0$ by~$\beta_{d, \alpha}(r)$ the power of the LR test of size~$\alpha$ against alternatives~$\bm{\theta} \in \R^d$ such that~$\|\bm{\theta}\|_2 = r$ (noting that the power of the LR test coincides for all such parameters as it is rotationally invariant).\footnote{With this notation it is worth noting that the proof of Theorem~\ref{thm:nogain2} shows that~$\epsilon$ in that theorem can be chosen as~$(1-\limsup_{d \to \infty} \beta_{d, \alpha_d}(r_d))/2,$ where~$\alpha_d$ denotes the size of~$\psi_d$.} Denote the set of all tests~$\psi: \R^d \to [0, 1]$ by~$\Psi_{d}$, and define for every~$\alpha \in [0, 1]$,~$\epsilon > 0$, and~$\psi \in \Psi_{d}$ the set~$\mathbb{F}_d(\epsilon, \psi)$ as the subset of parameters against which the power of~$\psi$ exceeds the power of the LR test of the same size as~$\psi$ by more than~$\epsilon$, i.e., 
\begin{equation}\label{eqn:inclconsset3}
\mathbb{F}_d(\epsilon, \psi) := \left\{\bm{\theta} \in \R^d: \mathbb{E}(\psi(\bm{\theta} + \bm{\varepsilon}_d)) - \beta_{d, \alpha}(\|\bm{\theta}\|_2) > \epsilon \text{ for } \alpha = \mathbb{E}(\psi(\bm{\varepsilon}_d)) \right\}.
\end{equation} 
The question is: how large can this set be made by cleverly choosing~$\psi$? The following theorem provides a non-asymptotic upper bound on its measure w.r.t.~to the uniform distribution~$\rho_{d, r}$ on~$\mathbb{S}_d^{d-1}(r)$. The upper bound decreases exponentially in~$d$.

\begin{proposition}\label{thm:nogain3}
For every~$\epsilon > 0$,~$r > 0$,~$d \in \N$ and~$\psi \in \Psi_{d}$, it holds that
\begin{align}\label{eqn:limmain}
\rho_{d, r} \left(  \mathbb{F}_d(\epsilon, \psi) \right) \leq 2 \exp\left( - \epsilon^2 (d-1)/ (2r^2) \right).
\end{align} 
\end{proposition}

\begin{remark}
Given a test~$\psi \in \Psi_d$, note that~$\mathbb{F}_d(\epsilon, \psi) \cap \mathbb{S}_d^{d-1}(r)$ is empty if~$\beta_{d, \alpha}(r) + \epsilon \geq 1$, where~$\alpha = \mathbb{E}(\psi(\bm{\varepsilon}_d))$. For such values of~$r$ and~$\epsilon$ it obviously holds that~$\rho_{d, r} \left(  \mathbb{F}_d(\epsilon, \psi) \right) = 0$. 
\end{remark}

The proof of Proposition~\ref{thm:nogain3} follows from the following ingredients: (i) L\'evy's concentration theorem, i.e., the fact that any Lipschitz continuous function on the sphere~$\mathbb{S}^{d-1}(r)$ concentrates around its average w.r.t.~$\rho_{d, r}$, see, e.g., Theorem 1.7.9 of \cite{braz}; (ii) the observation that power functions of tests in the model considered are Lipschitz continuous in the parameter vector; and (iii) the fact that the LR test maximizes (among all tests) the average power w.r.t.~$\rho_{d, r}$ against alternatives on the sphere~$\mathbb{S}^{d-1}(r)$. The proof of Theorem~\ref{thm:nogain2} is based on the inequality in Proposition~\ref{thm:nogain3}.

\section{Conclusion}

In high-dimensional testing problems, the choice of a test implicitly or explicitly determines the type of alternative it prioritizes. In the Gaussian sequence model, the LR test is based on the Euclidean norm. Many tests exist that are consistent against alternatives the LR test isn't consistent against (or are even consistent against strictly more alternatives than the LR test), i.e., they possess what we refer to as superconsistency points. We have shown that for any test, the corresponding set of superconsistency points is negligible in an asymptotic sense. This can be interpreted as a high-dimensional testing analogue of Le Cam's famous result that the set of superefficiency points relative to the maximum likelihood estimator is at most a Lebesgue null set, cf.~\cite{lecam1953}. In analogy to that classic finding, our result does not suggest that one should always use the LR test. But it shows that there exists no test for which one can expect substantial improvements.

\bibliographystyle{ims}
\bibliography{refs}

\appendix

\section{Proof of Theorem~\ref{thm:nogain2}}\label{proofmain0s}

Let the sequence of tests~$\psi_d$ and~$r_d$ be as in the theorem's statement. Denote the size of~$\psi_d$ by~$\alpha_d$. From~\eqref{eqn:inclconsset} we obtain
\begin{equation}\label{eqn:infproj}
c_d := \inf_{\bm{\theta} \in \mathbb{D}_d } \mathbb{E}\left( \psi_d(\bm{\theta} + \bm{\varepsilon}_d)\right)
\to 1.
\end{equation}
Since~$r_d/d^{1/4}$ is bounded and~$\alpha_d \to \alpha$, it follows from Theorem~\ref{thm:ing} that~$$\beta := \limsup_{d \to \infty} \beta_{d, \alpha_d}(r_d) < 1.$$ Define~$\epsilon = (1-\beta)/2$. For all~$d$ large enough we thus obtain~$c_d > \beta_{d, \alpha_d}(r_d) + \epsilon$. Together with~$\mathbb{D}_d \subseteq \mathbb{B}_2^d(r_d)$ and the function~$r \mapsto \beta_{d, \alpha_d}(r)$ being non-decreasing, it therefore follows that~$\mathbb{D}_d \subseteq \mathbb{F}_d(\epsilon, \psi_d)$ for all~$d$ large enough.\footnote{Throughout this proof we use the notation that was introduced in the context of Proposition~\ref{thm:nogain3}.} Proposition~\ref{thm:nogain3} hence allows us to conclude that for all~$d$ large enough, we have
\begin{equation}\label{eqn:sphD}
\rho_{d, r} \left(  \mathbb{D}_d \right) \leq \rho_{d, r} \left(  \mathbb{F}_d(\epsilon, \psi_d) \right) \leq 2 \exp\left( - \epsilon^2 (d-1)/ (2r^2) \right) \quad \text{ for every } r > 0.
\end{equation}
For every~$r > 0$, the push-forward measure of~$\rho_d := \rho_{d, 1}$ under the transformation~$\bm{\gamma} \mapsto r \bm{\gamma} $,~$\bm{\gamma}  \in \mathbb{S}^{d-1} := \mathbb{S}^{d-1}(1)$ is~$\rho_{d, r}$. Using polar coordinates (as in, e.g.,~\cite{stroock} Section 5.2) and~$\mathbb{D}_d \subseteq \mathbb{B}_2^d(r_d)$ we may express 
\begin{equation*}
\frac{\mathrm{vol}_d\left(\mathbb{D}_d\right)}{\mathrm{vol}_d\left(\mathbb{B}_2^d(r_d)\right)} = \frac{d}{r_d^d}
\int_{(0, r_d)} r^{d-1} \int_{\mathbb{S}^{d-1}} \mathds{1}_{\mathbb{D}_d}\{ r \bm{\gamma}  \} \mathsf{d} \rho_d(\bm{\gamma} ) \mathsf{d}  r = \frac{d}{r_d^d}
\int_{(0, r_d)} r^{d-1} \rho_{d, r} \left(  \mathbb{D}_d \right) \mathsf{d} r,
\end{equation*} 
which, for all~$d$ large enough, we can upper bound by
\begin{equation*}
\frac{2}{r_d^d}
\int_{(0, r_d)} d r^{d-1}\exp\left( - \epsilon^2 (d-1)/ (2r^2) \right) \mathsf{d} r \leq 2 \exp\left( - \epsilon^2 (d-1)/ (2r_d^2) \right).
\end{equation*}

\section{Proof of Proposition~\ref{thm:nogain3}}\label{proofmain02}

Let~$\psi \in \Psi_{d}$ and denote its size by~$\alpha$. If~$\alpha = 0$ or~$\alpha = 1$ the inequality in~\eqref{eqn:limmain} trivially holds as~$\mathbb{F}_d(\epsilon, \psi)  = \emptyset$ then follows. Hence, we only need to verify the claim for~$\alpha \in (0, 1)$. 

We next show that the LR test (for the testing problem~\eqref{eqn:tp} in the model~\eqref{eqn:model}) with size~$\alpha$ maximizes the ``weighted average power'' (WAP)
\begin{equation}
	\psi^* \mapsto \int_{\mathbb{S}^{d-1}(r)} \mathbb{E}(\psi^*(\bm{\gamma} + \bm{\varepsilon}_d)) \mathsf{d}  \rho_{d, r}(\bm{\gamma})
\end{equation}
among all tests~$\psi^* \in \Psi_d$ of size~$\alpha$: to see this, denote the $d$-variate normal density with mean~$\bm{\gamma}$ and identity covariance matrix by~$\phi_{\bm{\gamma}}$ and note that (by the Neyman-Pearson Lemma) the test which maximizes weighted average power (i.e., which is WAP optimal) is the likelihood ratio test for the simple hypothesis where (i) the density under the null equals~$\phi_{\bm{0}_d}$ and (ii) the density under the alternative equals~$\int \phi_{\bm{\gamma}} d\rho_{d, r}(\bm{\gamma})$. This test rejects for the observation~$y$ if and only if 
\begin{equation}\label{eqn:NPtest}
	\int_{\mathbb{S}^{d-1}(r)} \phi_{\bm{\gamma}}(y)/\phi_{\bm{0}_d}(y) \mathsf{d} \rho_{d, r}(\bm{\gamma}) = \exp(-r^2/2) \int_{\mathbb{S}^{d-1}(r)} \exp(\bm{\gamma}'y) \mathsf{d} \rho_{d, r}(\bm{\gamma})
\end{equation}
exceeds a critical value~$C_{\alpha, r}$, say, which is chosen such that~the test has size~$\alpha$.  Note that the measure~$\rho_{d, r}$ coincides with its push-forward measure under any orthonormal linear transformation~$U: \R^d \to \R^d$, i.e.,~$\rho_{d, r}$ is ``rotationally invariant". Choosing~$U$ orthonormal and such that~$U y$ coincides with~$\|y\|_2$ times the first element of the canonical basis of~$\R^d$, it follows that the integral to the right in~\eqref{eqn:NPtest} coincides with
\begin{equation*}
	\int_{\mathbb{S}^{d-1}(r)} \exp(\|y\|_2 \gamma_1) \mathsf{d} \rho_{d, r}(\bm{\gamma}) = \frac{1}{2} \left(\int_{\mathbb{S}^{d-1}(r)} \left[\exp(\|y\|_2 \gamma_1) + \exp(-\|y\|_2 \gamma_1)\right] \mathsf{d} \rho_{d, r}(\bm{\gamma})\right).
\end{equation*}
Since the function~$a \mapsto \exp(a \gamma_1) + \exp(- a \gamma_1)$ is non-decreasing on~$[0, \infty)$ for every~$\gamma_1$, it follows that the WAP optimal test rejects if and only if~$\|y\|_2$ exceeds a critical value (chosen so that the test has the right size). In other words the LR test (for the testing problem~\eqref{eqn:tp} in the model~\eqref{eqn:model}) is WAP optimal.

Because the LR test is WAP optimal~$\int_{\mathbb{S}^{d-1}(r)} \mathbb{E}(\psi(\bm{\gamma} + \bm{\varepsilon}_d)) \mathsf{d}  \rho_{d, r}(\bm{\gamma}) \leq \beta_{d, \alpha}(r).$ 
From this we can conclude that $\rho_{d, r} (  \mathbb{F}_d(\epsilon, \psi)) = \rho_{d, r} (  \mathbb{F}_d(\epsilon, \psi) \cap \mathbb{S}^{d-1}(r))$ is bounded from above by
\begin{equation}
	\rho_{d, r} \left( \left\{ \bm{\theta} \in \mathbb{S}^{d-1}(r): \mathbb{E}(\psi(\bm{\theta} + \bm{\varepsilon}_d)) \geq \int_{\mathbb{S}^{d-1}(r)} \mathbb{E}(\psi(\bm{\gamma} + \bm{\varepsilon}_d)) \mathsf{d}  \rho_{d, r} (\bm{\gamma}) + \epsilon \right\} \right),
\end{equation}
which we can equivalently express as
\begin{equation*}
	\rho_{d} \left( \left\{ \bm{\theta} \in \mathbb{S}^{d-1}: \mathbb{E}(\psi( r \bm{\theta} + \bm{\varepsilon}_d)) \geq \int_{\mathbb{S}^{d-1}} \mathbb{E}(\psi(r \bm{\gamma} + \bm{\varepsilon}_d)) \mathsf{d}  \rho_{d} (\bm{\gamma}) + \epsilon \right\} \right).
\end{equation*}
It is well-known that the total variation distance between two Gaussian distributions with mean vectors~$\bm{\theta}_1$ and~$\bm{\theta}_2$, respectively, is bounded from above by~$\|\bm{\theta}_1 - \bm{\theta}_2\|_2/2$. This implies that the power function~$\bm{\theta} \mapsto \mathbb{E}(\psi(\bm{\theta} + \bm{\varepsilon}_d))$ is Lipschitz continuous with constant~$1$.  Note that the function~$\bm{\theta} \mapsto \mathbb{E}(\psi(r \bm{\theta} + \bm{\varepsilon}_d))$ is then obviously Lipschitz continuous with constant~$r$. Since the geodesic distance between two points in~$\mathbb{S}^{d-1}$ is not smaller than the Euclidean distance between the two points, the function~$\bm{\theta} \mapsto \mathbb{E}(\psi(r \bm{\theta} + \bm{\varepsilon}_d))$ is Lipschitz continuous with constant~$r$ on~$\mathbb{S}^{d-1}$ when equipped with the geodesic distance. Using the concentration inequality for Lipschitz continuous functions on spheres in Theorem 1.7.9 of \cite{braz}, we obtain 
\begin{equation*}
	\rho_{d, r} \left(  \mathbb{F}_d(\epsilon, \psi) \right) \leq 2 \exp\left( - \epsilon^2 (d-1)/ (2r^2) \right).
\end{equation*}

\end{document}